\documentclass[a4paper]{amsart}
\usepackage[margin=3.5cm]{geometry}
\usepackage[utf8]{inputenc}
\usepackage{amsmath,amsfonts,amssymb,amsthm, multirow}
\usepackage{comment, stackrel}
\usepackage{hyperref}

\theoremstyle{plain}
\newtheorem{theorem}{Theorem}[section]

\newtheorem{proposition}[theorem]{Proposition}

\newtheorem{corollary}[theorem]{Corollary}

\theoremstyle{definition}

\newtheorem{example}[theorem]{Example}

\theoremstyle{remark}

\newcommand{\N}{\mathbb{N}}                    
\newcommand{\Z}{\mathbb{Z}}                    
\newcommand{\Q}{\mathbb{Q}}                    
\newcommand{\R}{\mathbb{R}}                    
\newcommand{\C}{\mathbb{C}}                    
\renewcommand{\P}{\mathbb{P}}                    %
\renewcommand{\O}{\mathcal{O}}                    %

\newcommand{\ba}{{\bf a}}
\newcommand{\bv}{{\bf v}}
\newcommand{\bw}{{\bf w}}
\newcommand{\be}{{\bf e}}

\newcommand{\Proj}{\operatorname{Proj}}

\newcommand{\im}{\operatorname{im}}

\numberwithin{equation}{section}

\title{Refined Ehrhart series and bigraded rings}

\author{Praise Adeyemo}
\address{Faculty of Science, University of Ibadan, Ibadan, Oyo State, Nigeria}
\email{ph.adeyemo@mail1.ui.edu.ng}

\author{Bal\'{a}zs Szendr\H{o}i}
\address{Mathematical Institute, University of Oxford, Oxford OX2 6GG, United Kingdom}
\email{szendroi@maths.ox.ac.uk}

\begin{document}
\begin{abstract} We study a natural set of refinements of the Ehrhart series of a closed polytope, first considered by Chapoton. We compute the refined series in full generality for a simplex of dimension $d$, a cross-polytope of dimension $d$, respectively a hypercube of dimension $d\leq 3$, using commutative algebra. We deduce summation formulae for products of $q$-integers with different arguments, generalizing a classical identity due to MacMahon and Carlitz. We also present a characterisation of a certain refined Eulerian polynomial in algebraic terms. 
\end{abstract}

\maketitle

\section*{Introduction} 

Given a $d$-dimensional lattice polytope $\Sigma$, the number of lattice points in its $r$-fold dilation is given by a polynomial in $r$, its Ehrhart polynomial. 
The generating function of the number of lattice points as $r$ varies gives rise to a related notion, the Ehrhart series of $\Sigma$. On the other hand, the polytope $\Sigma$
can also be used to define a projective toric variety $X_\Sigma$ which, under a further assumption, comes with a natural embedding into projective space. The Ehrhart polynomial 
and Ehrhart series of $\Sigma$ get identified with the Hilbert polynomial and Hilbert series of this embedding. 

Recently Chapoton~\cite{Chapoton} generalised Ehrhart theory by counting lattice points in slices defined by hyperplanes orthogonal to a fixed direction, giving a one-parameter
refinement of the Ehrhart series of $\Sigma$. 
In this short paper, we observe that Chapoton's refinement corresponds on the geometric side to a natural bigrading on the projective coordinate
ring of~$X_\Sigma$. We use this observation to calculate the refined Ehrhart series, for all possible choices of slicing direction, for the standard simplex and cross-polytope of arbitrary dimension~$d$, as well as for a
hypercube of dimension $d\leq 3$. Our main results are Propositions~\ref{prop_simplex}, \ref{prop:cross}, \ref{prop_two_terms} and~\ref{prop_three_terms}. We explain why a general formula for the refined Ehrhart series of a hypercube of dimension $d>3$ is not really of interest. We also use these ideas to formulate in Proposition~\ref{conj:P1P1} an interpretation of a refined Eulerian polynomial, due to MacMahon and Carlitz, in terms of commutative algebra. 

The projective coordinate ring of~$X_\Sigma$ is in fact naturally graded by $\Z\oplus\Z^d$, compare e.g.~\cite[Sect.13.3]{michalek_s}. It is straightforward to write out generalizations
of the formulae of this note in the multi-graded case. We chose to focus on the bigraded case, as that corresponds to Chapoton's refinement, and also that is 
the case that (sometimes) gives simplified expressions, such as the one corresponding to the refined Eulerian polynomial. 

\subsection*{Notation} We work in the fixed lattice $\Z^d$, with standard basis vectors $\be_1, \ldots, \be_d$ orthonormal with respect to the inner product $(\bv, \bw)\mapsto \bv\cdot\bw$. 
A $d$-dimensional lattice polytope $\Sigma\subset\R^d$ is a compact convex polytope with vertices contained in our lattice $\Z^d\subset\R^d$.

\subsection*{Acknowledgements} We would like to thank Miles Reid, John Sha\-reshian, Anna Seigal, Bernd Sturmfels and especially Mateusz Michalek for useful correspondence and comments. The ideas in this paper were formulated during a visit of P. A.~to the University of Oxford, supported by EPSRC GCRF grant EP/T001968/1, part of the Abram Gannibal Project. 

\section{Standard definitions and examples}

\subsection{Polytopes and Ehrhart theory}
Let $\Sigma\subset\R^d$ be a $d$-dimensional lattice polytope. For a non-negative integer~$r$, let $r\Sigma\subset\R^d$ denote the $r$-fold dilation of $\Sigma$. 
Then it is known that there exists a polynomial, 
the {\em Ehrhart polynomial} $L_\Sigma\in \Q[x]$, with the property that for all non-negative integers $r$, we have
\[ \#\{ r\Sigma \cap \Z^d\} = L_\Sigma(r).\]
We can also consider the {\em Ehrhart series}
\[ E_\Sigma(t) = \sum_{r=0}^\infty L_\Sigma(r) t^r. 
\]
Then it is known that
\[ E_\Sigma(t) = \frac{H_\Sigma(t)}{(1-t)^{d+1}},
\]
where $H_\Sigma(t)$ is a polynomial of degree $d$ with non-negative integer coefficients. 

\begin{example} \label{ex_simplex}
Let $\Sigma=\Delta_d\subset\R^d$ be the unit $d$-simplex spanned by the standard basis vectors of $\R^d$. The Ehrhart polynomial is
\[ L_{\Delta_d} (r) = {r+d \choose d},
\]
whereas the Ehrhart series is
\[ E_{\Delta_d}(t) = \frac{1}{(1-t)^{d+1}}.
\]
\end{example}

\begin{example}\label{ex_cross} Let $\Sigma=\Xi_d$ be the $d$-dimensional cross-polytope, the convex hull of all unit coordinate vectors and their opposites. 
Then (see for example~\cite[Thm.2.7]{beck}), the Ehrhart polynomial of $\Xi_d$ is
\[ L_{\Xi_d}(r) = \sum_{k=0}^d  {d \choose k} {r-k+d \choose d}.
\]
The corresponding Ehrhart series is
\begin{equation}  E_{\Xi_d}(r) = \frac{(1+t)^d}{(1-t)^{d+1}}.
\label{eq_cross_E}
\end{equation}
\end{example}

\begin{example} 
\label{ex_cube} Let $\Sigma=C_d$ be the unit hypercube of dimension $d$. The Ehrhart polynomial is clearly
\[ L_{C_d}(r) = (1+r)^d.
\]
Thus the Ehrhart series is 
\begin{equation}\label{eq_cu} E_{C_d}(t) = \sum_{r=0}^\infty (1+r)^d t^r  = \frac{A_d(t)}{(1-t)^{d+1}},
\end{equation}
where $A_d(t)$ is the $d$-th Eulerian polynomial, introduced by Euler to sum this very series. The first few examples are 
\[ A_2(t)=1+t, \ A_3(t)= 1 + 4t + t^2, \ A_4(t) =  1+ 11t + 11t^2 + t^3, \] and
\[ A_5(t) = 1 + 26 t + 66 t^2 + 26 t^3 + t^4.\]
All nontrivial coefficients of the Eulerian polynomial $A_d(t)$ are positive integers. In fact, as is well known,
\[ A_d(t)= \sum_{\sigma\in S_d} t^{\mathrm{desc}(\sigma)}, 
\] 
where for a permutation $\sigma\in S_d$ of $\{1,\ldots, d\}$, $\mathrm{desc}(\sigma)\in\N$ denotes the number of {\em descents} of~$\sigma$, the number of
positions $i < d$ with $\sigma(i) > \sigma(i+1)$. 
\end{example}

\subsection{Toric geometry of polytopes}

Given a convex lattice polytope $\Sigma\subset\R^d$, consider the $\C$-vector space
\[
R_\Sigma :=  \C\langle x_{\bv,r} \,|\, r\geq 0, \bv \in r\Sigma \cap \Z^d\rangle.
\]
This becomes a $\C$-algebra by the multiplication rule
\[
 x_{\bv,r}\cdot x_{\bv',r'} := x_{\bv+\bv', r+r'}.
\]
It is moreover graded, with $r$-th graded piece $R_{\Sigma,r}$ generated by the variables $x_{\bv,r}$ for $\bv \in r\Sigma \cap \Z^d$;
in particular $R_{\Sigma,0}\cong\C$. By definition, $E_\Sigma(t)$ is the Hilbert series of the graded algebra $R_\Sigma$. 

Recall that the lattice polytope $\Sigma$ is called {\em normal}, if for any $r>1$, every lattice point $\bv \in r\Sigma \cap \Z^d$ can be written as a vector sum $\bv=\sum_{i=1}^r \bv_i$ of $r$ lattice points
$\bv_i\in \Sigma \cap \Z^d$. 
If $\Sigma$ is normal, then $R_\Sigma$ is generated as a $\C$-algebra by its grade-$1$ piece
$R_{\Sigma,1}$. Let 
\[
S_\Sigma :=  \C[x_{\bv,1} \,|\, \bv \in \Sigma \cap \Z^d]
\]
be the {\em free} graded algebra on the lattice points in $\Sigma$. 

Consider $X_\Sigma=\Proj R_\Sigma$, a $d$-dimensional projective toric variety. If $\Sigma$ is normal, we have a toric embedding (compare~\cite[Chap. 14]{sturmfels})
\[ \nu_\Sigma\colon X_\Sigma = \Proj R_\Sigma \hookrightarrow \Proj S_\Sigma\cong \P^{\#\{ \Sigma \cap \Z^d\}-1}.
\]
In this case, the Ehrhart polynomial $L_\Sigma$ can be identified with the Hilbert polynomial of~$X_\Sigma$ in this embedding, and thus its Hilbert
series gets identified with the Ehrhart series of~$\Sigma$.  

\begin{example} For $\Sigma=\Delta_d\subset\R^d$ the standard $d$-simplex with $d+1$ vertices and no interior lattice point, the graded algebra $R_{\Delta_d}$ is the free polynomial algebra
\[R_{\Delta_d} = \C[x_{\bv,1} \mid  \bv \in \Delta_d \cap \Z^d]\cong \C[x_0, \ldots, x_d] \]
generated by variables corresponding to the origin, respectively the other vertices.
The attached projective embedding is $\nu_{\Delta_d}\colon X_{\Delta_d}\cong\P^d$.
\end{example}

\begin{example} \label{ex_toric_cube}
Consider the unit hypercube $\Sigma=C_d$ of dimension $d$, with $2^d$ vertices and no interior lattice point. As this is also normal, the graded algebra $R_{C_d}$ is still generated by degree one variables. 
It is known that the surjection $S_{C_d}\rightarrow R_{C_d}$ has a kernel
\[I_{C_d} = \langle x_{\bv,1} x_{\bv',1}-x_{\bw,1}x_{\bw',1}: \bv+\bw' = \bw + \bv'\rangle \lhd S_{C_d}
\]
generated by quadratic relations. We get the embedding 
 \[\nu_d\colon X_{C_d} \cong (\P^1)^d\hookrightarrow \Proj S_d\cong\P^{2^d-1}\] 
of the projective toric variety $(\P^1)^d$ by its complete linear series $\O(1,\ldots, 1)$. 

It will be important below to spell out some special cases. For $d=2$, we have the unit square as our polytope. Re-labelling degree-$1$ variables for easier readability, the algebra~$R_{C_2}$ is
\[ R_{C_2} \cong \C[x_{00},x_{01},x_{10},x_{11}]/ \langle x_{00}x_{11}-x_{01}x_{10}\rangle,
\]
corresponding to the Segre embedding $\nu_{C_2}\colon (\P^1)^2\hookrightarrow\P^3$, with Ehrhard--Hilbert series
\[E_{C_2}(t) =  \frac{1+ t}{(1-t)^3} =\frac{1-  t^2}{(1-t)^4}.
\]

For $d=3$, and $\Sigma=C_3$, the unit cube, the algebra has a more interesting structure. Again, re-labelling the variables in an obvious way, we have
\[ R_{C_3} \cong \C[x_{000},x_{001},x_{010},x_{011}, x_{100},x_{101},x_{110},x_{111}]/ I_{C_3},
\]
where $I_{C_3}$ a codimension $4$ ideal, the ideal of the embedding $\nu_{C_3}\colon (\P^1)^3\hookrightarrow\P^7$. It has a $9\times 16$ Gorenstein resolution: the $8$ generators satisfy $9$ quadric relations related by $16$ cubic syzygies, as evidenced by the Ehrhart--Hilbert series
\begin{equation}E_{C_3}(t) =  \frac{1+ 4t + t^2}{(1-t)^4} =\frac{1- 9 t^2 + 16 t^3-9t^4+t^6}{(1-t)^8}.
\label{ringC3}
\end{equation}
\end{example}

\begin{example} 
Finally consider $\Sigma=\Xi_d$ be the $d$-dimensional cross-polytope, once again a normal polytope, with $2d$ vertices and one interior lattice point, the origin. 
Then it is clear from basic linear algebra that all relations between the degree-$1$ variables are generated by the relations
\[ x_{\be_j, 1} x_{-\be_j,1} = x_{{\mathbf 0},1}^2,\]
one for each basis vector $\be_j\in\R^d$.  Re-labellling the degree-$1$ variables as $x_0$ for the origin, $x_{01}, \ldots, x_{0d}$ for the positive basis vectors, and $x_{11}, \ldots, x_{1d}$ for the negative ones, we thus have 
\[ R_{\Xi_d} \cong \C[x_{0},x_{01},\ldots, x_{0d}, x_{11},\ldots x_{1d}]/ \langle x_{0j}x_{1j}-x_0^2 : 1\leq j \leq d\rangle.
\] 
The corresponding toric embedding is 
\[\nu_{\Xi_d} \colon X_{\Xi_d} \hookrightarrow \Proj \C[x_{0},x_{01},\ldots, x_{0d}, x_{11},\ldots x_{1d}]\cong \P^{2d}
\]
presenting $X_{\Xi_d}$ as a complete intersection of $d$ quadric hypersurfaces in $\P^{2d}$. The Ehrhart--Hilbert series is now
\[  E_{\Xi_d}(r) = \frac{(1+t)^d}{(1-t)^{d+1}} =  \frac{(1-t^2)^d}{(1-t)^{2d+1}} .
\]
In the special case $d=2$, we get 
\[\nu_{\Xi_2} \colon X_{\Xi_2} \cong\left\{   x_{01}x_{11}= x_{02}x_{12}=x_0^2 \right\} \subset  \P^4,
\]
a canonical del Pezzo surface of degree $4$ with four $A_1$ singularities. 
\end{example}

\section{The refinement}

\subsection{The definition}

We follow Chapoton~\cite{Chapoton}. Fix a $d$-tuple of integer weighs ${\bf a}=(a_1, \ldots, a_d)$. 
For integers $k$, consider the quantity
\[ L^\ba_\Sigma(r,k) = \#\{ v \in r\Sigma \cap \Z^d : \bv\cdot {\ba} =k \}.\]
In other words, we are slicing the polytope with hyperplanes perpendicular to the direction~$\ba$, and counting lattice points in the different layers. We can then set
\[ L^\ba_{\Sigma,r}(q) = \sum_{k\in\Z} L_\Sigma(r,k)q^k,\]
a Laurent polynomial in a new variable $q$, and consider the series
\[ E^\ba_{\Sigma}(t,q) = \sum_{r=0}^\infty \sum_{k\in\Z} L_\Sigma(r,k) \, q^k t^r.
\]
Setting $q=1$, these specialise to the value at $r$ of the Ehrhart polynomial and the Ehrhart series, respectively. Chapoton  \cite{Chapoton} proves that under a non-negativity assumption as well as genericity conditions on the weights, 
the refined Ehrhard series is a rational function in $(t,q)$ with a controlled denominator. We will not need these non-negativity and genericity assumptions in what follows. 

From the point of view of the algebra $R_\Sigma$, the weights furnish it with a bigrading, with the generator $x_\bv$ for $\bv \in r\Sigma \cap \Z^d$ having degree $(r,\bv\cdot \ba)$. By definition, the series $E^\ba_{\Sigma}(t,q)$ is the Hilbert series of this bigraded algebra. 

\subsection{The simplex} Return to $\Sigma=\Delta_d$, the unit $d$-simplex. Choose first $\ba=(1,\ldots, 1)$. This is a non-generic choice in the sense of~\cite{Chapoton}: we are slicing the $d$-simplex into parallel regular $(d-1)$-simplices. We clearly have
\[ L^{(1,\ldots, 1)}_{\Delta_d}(r,k) = {k+d-1 \choose d-1} \mbox{ for } 0\leq k \leq r.\]
The polynomial
\[
L^{(1,\ldots, 1)}_{\Delta_d,r}(q) = \sum_{k=0}^r {k+d-1 \choose d-1}q^k 
\]
specialises at $q=1$ to the binomial coefficient ${r+d \choose d}$ by standard rules of the Pascal triangle. 
It easy to compute~\cite{Adeyemo}, by elementary methods, the corresponding refined Ehrhart series
\[E^{(1,\ldots, 1)}_{\Delta_d}(t,q) = \frac{1}{(1-t)(1- qt)^{d}}.
\]

Next choose $\ba=(1,2, \ldots, d)$, a generic choice of slicing. There is no short expression for the quantity $L^{(1,2, \ldots, d)}_{\Delta_d}(r,m)$. On the other hand, the polynomial
\begin{equation} L^{(1,2, \ldots, d)}_{{\Delta_d},r}(q) =\left[ r+d \choose d\right]_q \mbox{ for } 0\leq k \leq r
\label{eq:gauss}
\end{equation}
is known to be~\cite{Stanley, Adeyemo} a $q$-binomial coefficient. So in this case, the substitution $q=1$ specialises the $q$-binomial coefficient to the standard one. 


It is easy to write down the refined Eckhart series of the $d$-simplex in full generality.

\begin{proposition} For an arbitrary $d$-tuple of integers ${\bf a}=(a_1, \ldots, a_d)$, we have
\[E^\ba_{\Delta_d}(t,q) = (1-t)^{-1}\prod_{j=1}^d(1-q^{a_j}t)^{-1}.\]
\label{prop_simplex}
\end{proposition}
\begin{proof}  As noted before, the $\C$-algebra $R_{\Delta_d}$ is freely generated by the variables $x_0, x_1, \ldots, x_d$ corresponding to the lattice points
${\mathbf 0}, \be_1, \ldots, \be_d$ of $\Delta_d$. The bigraded degree of these generators is $(1,0)$, respectively $(1, a_j)$ for $j=1, \ldots, d$. The statement
follows. 
\end{proof}

For the specific weights $\ba=(1,2, \ldots, d)$, we get 
\[E^{(1,2, \ldots, d)}_{\Delta_d}(t,q) = \prod_{j=0}^d(1-q^{j}t)^{-1} = \sum_{r=0}^\infty  \left[ k+d-1 \choose d-1\right]_q  t^r
\]
where the second equality is one form of the $q$-binomial theorem, agreeing with~\eqref{eq:gauss}.

\subsection{The cross-polytope} Let $\Sigma=\Xi_d$, the $d$-dimensional cross-polytope. The next result is also straightforward. 
\begin{proposition} \label{prop:cross}
For an arbitrary $d$-tuple of integers ${\bf a}=(a_1, \ldots, a_d)$, we have
\begin{equation}
E^\ba_{\Xi_d}(t,q) = \frac{(1+t)^d(1-t)^{d-1}}{\displaystyle\prod_{j=1}^d(1-q^{a_j}t)\prod_{j=1}^d(1-q^{-a_j}t)} .\label{eq:crosspol}\end{equation} 
\label{eq_simplex}
\end{proposition}
\begin{proof}  As discussed in Example~\ref{ex_cross}, the $\C$-algebra $R_{\Xi_d}$ is generated by $x_0$ of bigraded degree $(1,0)$, and generators $x_{0j}, x_{1j}$ of bigraded degree $(1, a_j)$, respectively $(1, -a_j)$ for $j=1, \ldots, d$. There are also $d$ relations of bidegree $(2,0)$ which form a regular sequence. The statement follows after factoring the numerator in the resulting Hilbert series, and cancelling one factor of $(1-t)$.  
\end{proof}
Notice that the only time the right hand side of~\ref{eq:crosspol} simplifies further is if there exists $1\leq j \leq d$ with $a_j=0$. Apart from the trivial case $\ba={\mathbf 0}$, the only time all negative terms in the numerator cancel is 
\[
E^{(0,\ldots,0,1)}_{\Xi_d}(t,q) = \frac{(1+t)^d}{(1- qt)(1-q^{-1}t)(1-t)^{d-1}}.\] 
In this case, we do {\em not} get $q$-deformations of the numerator of~\eqref{eq_cross_E}. Indeed, recent work of Tielker~\cite{Tielker} uses a different, non-linear 
refined count of lattice points inside dilations of the cross-polytope to get an interesting $q$-refinement of its Ehrhart numerator. Our last example will behave very differently in this regard. 

\subsection{The hypercube} 

Finally consider $\Sigma=C_d$, the unit $d$-hypercube. At the $r$-th dilation, we have a hypercube of side length $r$, whose $(1+r)^d$ lattice points are weighted by $a_j$ in direction~$j$, so
\[ L^\ba_{C_d,r}(q) = \prod_{j=1}^d \left( \sum_{i=0}^r  q^{ia_j}\right).\]
Summing each of these geometric series, we deduce
\begin{equation} \label{cube_formula} E^\ba_{C_d}(t,q) = \sum_{r=0}^\infty \left( \prod_{j=1}^d [r+1]_{q^{a_i}}\!\!\right) t^r,  \end{equation}
where for a positive integer $k$, 
\[ [k]_q = \frac{1-q^{k}}{1-q}
\]
is the corresponding $q$-integer. 

For symmetric weights $\ba=(1,\ldots, 1)$, we get
\begin{equation} \label{symm_cube_formula} E^{(1,\ldots, 1)}_{C_d}(t,q) = \sum_{r=0}^\infty [r+1]_{z}^d \, t^r=  \frac{\tilde A_d(t,q)}{\prod_{j=0}^d (1-q^jt)}.\end{equation}
Here the second equality is  a formula of MacMahon and Carlitz~\cite{MacMahon, Carlitz, Chow-Gessel}, and involves the two-variable refined Eulerian polynomial 
\[\tilde A_d(t,q) = \sum_{\sigma\in S_d} q^{\mathrm{maj}(\sigma)}t^{\mathrm{desc}(\sigma)}.
\] 
This enumerates permutations $\sigma\in S_d$ according to descent and major index, where the latter is defined as the sum of
all the descent positions in $\sigma$. 
The first few examples of this refined Eulerian polynomial are
\[ \tilde A_2(t, z)=1+ qt, \  \tilde A_3(t,q)= 1 + 2(1  + q)q  t  + q^3t^2, \]
and
\[
\tilde A_4(t,q) = 1 + (3 + 5 q + 3 q^2) qt  + (3  + 5  q + 3  q^2) q^3t^2 +  q^6 t^3. 
\]
At $q=1$, the refined Eulerian polynomial specialises to the original one, while \eqref{symm_cube_formula} specialises to \eqref{eq_cu}. 

To evaluate more cases of~\eqref{cube_formula}, consider first $d=2$, so work with $\Sigma=C_2$, the unit square. Recall that in this case, the algebra $R_{C_2}$ can be described as
\[ R_{C_2} \cong \C[x_{00},x_{01},x_{10},x_{11}]/ \langle x_{00}x_{11}-x_{01}x_{10}\rangle.
\]
For a refinement with weights $\ba=(a,b)$, 
the bigraded algebra $R_{C_2}$ is generated by four generators of bidegrees $(1,0)$, $(1,a)$, $(1,b)$ and $(1,a+b)$ with a single relation in bidegree $(2, a+b)$. We deduce
\begin{proposition} For arbitrary integer weights $\ba=(a,b)$, we have
\[E^{(a,b)}_{C_2}(t,q) = \frac{1-q^{a+b}t^2}{(1-t)(1-q^{a}t)(1-q^{b}t)(1-q^{a+b}t)} .\]
\label{prop_two_terms}
\end{proposition}
Note that the right hand side does {\em not} simplify for general $\ba$. However, for $\ba=(1,1)$, there is an obvious cancellation of one term, and we get \[E^{(1,1)}_{C_2}(t,q) = \frac{1+ qt}{(1-t)(1- qt)(1-q^{2}t)} ,\]
the expected $q$-refinement of 
\[E_{C_2}(t) = \frac{1+t}{(1-t)^3}.
\]
On the other hand, for arbitrary weights, we deduce the following identity, a generalization of the classical MacMahon--Carlitz identity~\eqref{symm_cube_formula}.
\begin{corollary} For arbitrary integers $(a,b)$, we have 
\[ \sum_{r=0}^\infty  [r+1]_{q^{a}}  [r+1]_{q^{b}} t^r= \frac{1-q^{a+b}t^2}{(1-t)(1-q^{a}t)(1-q^{b}t)(1-q^{a+b}t)}.\]
\end{corollary}

Next let $\Sigma=C_3$, the unit cube. Recall that here we have 
\[ R_{C_3} \cong \C[x_{000},x_{001},x_{010},x_{011}, x_{100},x_{101},x_{110},x_{111}]/ I_{C_3}, 
\]
with $I_{C_3}$ a codimension $4$ ideal discussed in Example~\ref{ex_toric_cube}. For an arbitrary refinement with weights $\ba=(a,b,c)$, we get the following. 
\begin{proposition} \label{prop_three_terms}
\[E^\ba_{C_3}(t,q) =  \frac{N^\ba_{C_3}(t,q)}{(1-t) \left(1- q^at\right) \left(1-q^bt\right) \left(1-q^ct\right) \left(1-q^{a+b}t\right) \left(1-q^{a+c}t\right) \left(1-q^{b+c}t\right) \left(1-q^{a+b+c}t\right)},
\]
where the numerator is
\setlength{\arraycolsep}{2pt}

{\footnotesize
\[ \begin{array}{ll}
N^\ba_{C_3}(t,q)= & 1 \\ & -  \left(3 q^{a+b+c}+q^{2 a+b+c}+q^{a+2 b+c}+q^{a+b+2 c}+q^{a+b}+q^{a+c}+q^{b+c}\right)t^2 \\ &  + \left(2 q^{a+b+c}+2 q^{2 a+b+c}+2 q^{a+2 b+c}+2 q^{a+b+2 c}+2 q^{2 a+2 b+c}+2 q^{2 a+b+2 c}+2 q^{a+2 b+2 c}+2 q^{2(a+b+c)}\right) t^3 \\ & -\left(q^{2 a+2 b+c}+q^{2 a+b+2 c}+q^{a+2 b+2 c}+3 q^{2(a+ b+ c)}+q^{3 a+2 b+2 c}+q^{2 a+3 b+2 c}+q^{2 a+2 b+3 c}\right) t^4  \\ & + q^{3 (a+b+c)}t^6.\end{array} \]}
\end{proposition}
\begin{proof} This can easily be obtained from the known (or machine computable) resolution of the ideal $I_{C_3}\lhd \C[x_{000},x_{001},x_{010},x_{011}, x_{100},x_{101},x_{110},x_{111}]$, 
which is compatible with the bigrading. We omit the details.  
\end{proof}

From this formula, it is easy to derive the specialisation $\ba=(1, 1, 1)$ corresponding to symmetric weights:
\[E^{(1,1,1)}_{C_3}(t,q) = \frac{1 + 2q(1  + q) t + q^3 t^2 }{(1-t)(1- qt)(1-q^2t)(1-q^3t)},\]
recovering the refined Eulerian polynomial $\tilde A_3(t,q)$ in the numerator. 
A slightly more general, but attractive, specialisation is $\ba=(a,1,1)$ which gives the following. 
\begin{corollary} For any integer $a$, we have
\[ \sum_{r=0}^\infty  [r+1]_{q^{a}}  [r+1]_{q}^2 \, t^r= \frac{ 1+ q \left(q^a+1\right)t-2  \left(q^2+q+1\right) q^{a+1}t^2+\left(q^a+1\right) q^{a+3}t^3 +q^{2 a+4}t^4}{(1-t) (1- qt) \left(1- q^2t\right) \left(1-q^at\right) \left(1- q^{a+1}t\right) \left(1-q^{a+2}t\right)}.\]
\end{corollary}

For $d\geq 4$, the graded algebra $R_{C_d}=S_d/I_{C_d}$ corresponding to the embedding $\nu_d\colon (\P^1)^d\hookrightarrow\P^{2^d-1}$ is getting more and more 
complicated. Already for $d=4$, one has a codimension $11$ embedding, defined by $55$ quadrics; the formula corresponding to that of Proposition~\ref{prop_three_terms}, even if computable, would have thousands of terms. 

Turning the ideas around, let us see whether we can learn anything about the ideal~$I_{C_d}$ of the embedding $\nu_d$ in general. For general~$d$, formula~\eqref{ringC3} becomes 
\[E_{C_d}(t) =  \frac{A_d(t)}{(1-t)^{d+1}} =\frac{1-  c_d t^2 + O(t^3)}{(1-t)^{2^d}},
\]
for a certain positive integer $c_d$. On the right hand side, we have the Hilbert series of a variety embedded in $\P^{2^d-1}$, whose ideal is cut out by $c_d$ quadrics. In the middle, remembering that all coefficients of $A_d(t)$ are positive, we have the Hilbert series of a finite length module over a free polynomial algebra in $d+1$ generators, corresponding to a free subalgebra of $R_{C_d}$ formed of elements of a regular sequence; this is basically a form of Noether normalization, remembering that the image variety is projective of dimension $d$. 
See~\cite[Section 4]{stanley_CM} for more details on these ideas in the present context. 

The fact that the formula of Proposition~\ref{prop_two_terms} does not simplify for 
general weights tells us that this procedure is not compatible with general bigradings even for $d=2$. The fact that there is simplification for symmetric weights suggests that it is compatible with that
particular bigrading. On the other hand, we have the general MacMahon--Carlitz formula
\[E^{(1,\ldots, 1)}_{C_d}(t,q) = \frac{\tilde A_d(t,q)}{\prod_{j=0}^d (1-q^jt)},\]
with numerator with positive coefficients only. This suggests that for all $d$, $R_{C_d}$ should be finite as a bigraded module over a bigraded subalgebra generated by a regular sequence of elements of bidegrees $(1,0), (1,1), \ldots, (1,d)$. 

To find this subalgebra explicitly, for $0\leq k \leq d$ let 
\[ y_k = \sum_{\substack{\bv \in C_d\cap\Z^d \\ \bv\cdot (1,...,1)=k}} x_{\bv,1}\in S_d.
\] 
These elements are bihomogeneous of degree $(1,k)$ for the symmetric bigrading. 

\begin{proposition} \label{conj:P1P1}
The images $[y_0], \ldots, [y_d]\in R_{C_d}$ form a regular sequence. 
The quotient
\[T_{C_d} = R_{C_d} /\langle [y_0], \ldots, [y_d]\rangle
\]
is a finite dimensional bigraded $\C$-vector space with Hilbert series equal to the refined Eulerian polynomial $\tilde A_d(t,q)$.
\end{proposition}
\begin{proof}\!\!\footnote{We would like to thank Mateusz Michalek for suggesting this proof.} The graded ring $R_{C_d}$ is Cohen--Macaulay (in fact even Gorenstein). 
Thus in order to show that its elements $[y_0], \ldots, [y_d]\in R_{C_d}$ form a regular sequence, it is enough to show that the ideal they generate has codimension $(d+1)$. 
The latter is equivalent to the fact that the projective intersection of $\im\nu_{C_d}$ with the $d+1$ hyperplanes in $\P^{2^d-1}$ given by the 
conditions $y_k=0$ for $0\leq k \leq d$ is empty. Assume otherwise: 
choose a point $([a_1:b_1],[a_2:b_2],...,[a_d:b_d])\in(\P^1)^d$ whose image under $\nu_{C_d}$ is on the vanishing locus of all these hyperplanes. 
The vanishing of the first linear form tells us that the product of all the $a_i$ is zero; we may assume $a_1=0$, and hence $b_1\neq 0$.
The vanishing of the second linear form tells us that the sum of $d$-products of all but one of the $a_i$ and the complimentary $b_j$ is~0. Given that $a_1=0$, we deduce
$b_1a_2a_3\cdots a_d=0$, so without loss of generality $a_2=0$ and so $b_2\neq 0$. Continuing in this way, we get that on the vanishing locus of all the hyperplanes, all the
$a_i$ coordinates are zero. But the vanishing of the last linear form tells us that the product of all the $b_i$ is also zero, which is a contradiction. This proves the first statement. 

It follows in particular that $R_{C_d} /\langle [y_0], \ldots, [y_d]\rangle$ is a finite dimensional bigraded vector space. By the MacMahon--Carlitz formula, its bigraded Hilbert series 
indeed equals the refined Eulerian polynomial $\tilde A_d(t,q)$.
\end{proof}


\end{document}